\pdfoutput=1
\documentclass[11pt]{article}

\usepackage[a4paper,hmargin=1in,vmargin=1.25in]{geometry}
\usepackage[utf8]{inputenc}
\usepackage[toc,page]{appendix}
\usepackage{multicol}
\usepackage{amsmath,amsfonts,amssymb,amsthm,url}
\usepackage{tikz}
\usepackage{longtable}
\usepackage{booktabs,authblk}
\usepackage{multirow}
\usepackage{threeparttable}
\usetikzlibrary{shapes,arrows}
\usetikzlibrary{backgrounds}

\usepackage{caption}
\usepackage{subcaption}

\usepackage{longtable}
\usepackage{mathtools}

\def\Real{\mathbf{R}} 

\def\Mat{M} 

\def\grse{E}  
\def\sgrse{S} 
\def\indse{J} 
\def\iindse{P} 
\def\sindse{I} 

\def\field{\mathbf{F}}
\def\kos{K}

\def\sss{\smallsetminus}

\theoremstyle{plain}
\newtheorem{proposition}{Proposition}[section]

\newtheorem{lemma}[proposition]{Lemma}

\theoremstyle{definition}
\newtheorem{definition}[proposition]{Definition}

\title{Efficient Representation of Lattice Path Matroids}
\author{Carles Padr\'o}
\affil{Universitat Polit\`ecnica de Catalunya, Barcelona, Spain}

\begin{document}

\maketitle

\begin{abstract}
Efficient deterministic algorithms to construct 
representations of lattice path matroids
over finite fields are presented.
They are built on known constructions
of hierarchical secret sharing schemes,
a recent characterization of hierarchical matroid ports, 
and the existence of isolating weight functions
for lattice path matroids whose values are polynomial 
on the size of the ground set.
\end{abstract}

\section{Introduction}
\label{sec:introd}

Every linear code determines a matroid, namely
the one that is represented by any of its generator matrices.
Several properties of the code are derived from that matroid.
For example, the weight enumerator of the code
is derived from the Tutte polynomial 
of the matroid~\cite{Gre76}.
Another example is the correspondence between 
maximum distance separable linear codes
and representations of uniform matroids.
Determining over which fields uniform matroids
are represented is equivalent to 
solving the  main conjecture for
maximum distance separable codes.
Detailed information on that conjecture
is given in~\cite[Problem~6.5.19]{Oxl11}.
Important advances for solving it
have been presented in~\cite{Bal12,BaBe12}.

Additional connections follow from applications of
linear codes other than plain error detection and correction,
as network coding~\cite{DFZ11} or 
locally repairable codes~\cite{CTL20,TPD16}.
Among them, secret sharing has attracted most attention, 
and it is the main motivation for our results.
Vector secret sharing schemes~\cite{Bri89}
are ideal and linear, so they are among 
the most efficient secret sharing schemes.
The access structures of  
vector secret sharing schemes
coincide with the ports of representable matroids.
Every representation of a matroid
over a finite field provides
a vector secret sharing scheme for
each of its ports.
The efficiency of those schemes
is determined by the size of the field.

Given a family of representable matroids, 
some basic questions are
motivated by those applications.
Over which (finite) fields can 
they be represented?
Which is the minimum size of those fields?
Are there efficient algorithms 
to find a representation for 
every member of the family?

This paper deals with 
efficient deterministic constructions
of representations of matroids over finite fields.
Specifically, deterministic algorithms
that provide, for each member in a given family
of representable matroids, a representation
over some finite field $\field_q$.
Both the running time and
the size $\log q$ of the elements in the finite field
must be polynomial in the number of elements 
in the ground set.

The existence of such algorithms is well known
for uniform and graphic matroids,
and it has been proved for matroids with two 
clonal classes~\cite{BPWX12}.
Efficient deterministic representations
for other families of matroids are derived 
from constructions of several classes of
vector secret sharing schemes,
namely hierarchical~\cite{Bri89,CTL22,Tas07},
compartmented~\cite{CTL19},
and uniform multipartite~\cite{CRHC24} schemes.
Some results on deterministic algorithms
for transversal matroids are given
in~\cite{LMPSZ18,MPRS20},
but no efficient algorithms are known for that class.

For other families of matroids,
only efficient randomized algorithms are known.
This is the case for transversal matroids.
The output of those algorithms 
is correct with high probability,
but there is no efficient way to check
whether or not this is the case.

In this paper, we present efficient deterministic 
algorithms for lattice path matroids,
a family of transversal matroids
introduced in~\cite{BMN03}.
Even though the existence of such algorithms
has not been explicitly stated before,  
it directly follows from
previous works in secret sharing.
Namely, the constructions of 
hierarchical vector secret 
sharing schemes in~\cite{Bri89,CTL22,Tas07}
and a recent characterization of the matroids
determined by those schemes,
which were proved to coincide with
lattice path matroids~\cite{Mo23}.
In addition to pointing out
and explaining that connection, 
the main contribution in this paper
is a simpler and more general description
of those constructions.
Specifically, by using isolating weight functions
in a similar way as in~\cite{LMPSZ18},
a general method to 
find representations of transversal matroids
is presented.
It is efficient if the values 
of the isolating weight functions
are polynomial in the size of the ground set.
The existence of such functions
is proved for lattice path matroids.
The application to other
families of transversal matroids remains an open question. 

Our algorithms provide two kinds of 
representations of lattice path matroids.
The first one works for
large algebraic extensions of relatively small prime fields
and it corresponds to the constructions of
hierarchical secret sharing schemes in~\cite{Bri89,CTL22}.
The second one deals with
large prime fields.
Another such construction was proposed in~\cite{Tas07},
but it applies only to nested matroids.

\section{Preliminaries}

The main concepts and known results
together with the terminology and notation
that are used in the paper are explained in this section.
After presenting the basics on
polymatroids and matroids, we discuss
the connection between transversal 
matroids and Boolean polymatroids,
which is used later to describe
the two existing characterizations~\cite{FaPa10,Mo23} of
hierarchical matroid ports,
that is, the access structures 
of ideal hierarchical secret sharing schemes.
In addition, we review some facts and open questions
on ideal multipartite secret sharing schemes
and multi-uniform matroids 
that are the initial point of this work.
Polymatroids play a fundamental role in that topic.

\subsection{Polymatroids, Matroids, and Matroid Ports}

The reader is  referred to~\cite{Oxl11}
for a textbook on matroid theory.
Most of the time we use here the terminology
and notation from it. 
More information about
polymatroids, matroid ports,
and their application to secret sharing
is found in~\cite{MaPa10}. 

A \emph{set function} $f \colon 2^{\grse} \to \Real$
on a finite set $\grse$ is \emph{monotone}
if $f(X) \le f(Y)$ whenever 
$X \subseteq Y \subseteq \grse$,
and it is \emph{submodular} if
\(
f(X) + f(Y) - f(X \cup Y) - f(X \cap Y) \ge 0
\)
for all $X,Y \subseteq \grse$.
A \emph{polymatroid} is a pair
$(\grse,f)$ formed by a  \emph{ground set}
$\grse$ and a \emph{rank function} $f$.
The former is a finite set
and the latter is a 
monotone, submodular set 
function on $\grse$ with $f(\emptyset) = 0$.
\emph{Integer polymatroids} are those with
integer-valued rank functions.
From now on, only integer polymatroids are considered.

An integer polymatroid $\Mat = (\grse,r)$ with
$r(\{x\}) \le 1$ for each $x \in \grse$
is a \emph{matroid}.
The \emph{independent sets}
of the matroid $\Mat$ are the
subsets of the ground set with
$r(X) = |X|$.
A \emph{basis} is a maximal independent set
and a \emph{circuit} is a minimal dependent set.
All bases have $r(\grse)$ elements, 
and that value is called the \emph{rank} of the matroid. 

For an element $p_o$ in the ground set $\grse$,
the \emph{port of the matroid $\Mat$ at $p_o$}
is formed by the sets 
$X \subseteq E \sss \{p_o\}$
such that $r(X \cup \{p_o\}) = r(X)$.
Observe that the minimal sets in the
matroid port are the ones
such that $X \cup \{p_o\}$ is a circuit. 
A matroid is \emph{connected}
if every two different elements
in the ground set lie in a common circuit.
As a consequence
of~\cite[Theorem~4.3.3]{Oxl11}, a 
connected matroid is determined 
by any of its ports.

Given a matrix $A$ over a field $\kos$
with columns indexed by a set $\grse$
and a set  $X \subseteq \grse$,
let $r(X)$ be the rank of the submatrix
formed by the columns corresponding to 
the elements in $X$. 
Then $\Mat = (E,r)$ is a matroid.
In that situation,
$\Mat$ is \emph{representable} over $\kos$,
or $\kos$-representable, and the matrix $A$ is a 
\emph{representation} of $\Mat$ over $\kos$. 

While representations of matroids are 
collections of vectors (the columns of a matrix)
some polymatroids can be represented by
collections of vector subspaces.
A polymatroid $(E,f)$ is 
\emph{$\kos$-representable} if there exists  a collection 
$(V_x)_{x \in \grse}$
of subspaces of a $\kos$-vector space $V$ such that
\(
f(X) = \dim 
\sum_{x \in X} V_x
\)
for every $X \subseteq E$.

\subsection{Transversal Matroids and Boolean Polymatroids}

We discuss next some basic facts about
transversal matroids, Boolean polymatroids,
and lattice path matroids.
The reader is referred 
to~\cite{Bon10,BoMi06,BMN03,Mat01,Oxl11}
for additional information on those topics.

For an integer polymatroid 
$(\sgrse,f)$, consider the family 
formed by the subsets $X \subseteq \sgrse$ such that 
$|Y| \le f(Y)$ for every $Y \subseteq X$.
By~\cite[Corollary~11.1.2]{Oxl11}, 
that is the family of independent sets
of a matroid, which is
called the \emph{matroid induced by 
the polymatroid $(\sgrse,f)$}.

Let $G$ be a bipartite
graph with vertices in the parts
$\indse$
and $\sgrse$.
For a set $X$ of vertices,
$N(X)$ denotes the set of neighbors of the vertices in $X$.
If $B \subseteq \sgrse$, we
notate $G_B$ for the subgraph of $G$
induced by $\indse \cup B$.
The \emph{biadjacency matrix} of
the bipartite graph $G$ is a
$(0,1)$-matrix whose rows and columns are indexed 
by the sets $\indse$ and $\sgrse$, respectively, 
and the entries equal to $1$ mark the edges of $G$.

The graph $G$ determines two sequences of sets.
Namely, $(C_x \,:\, x \in \sgrse)$ with 
$C_x = N(\{x\}) \subseteq \indse$
and $(A_j \,:\, j \in \indse)$ with 
$A_j = N(\{j\}) \subseteq \sgrse$.
Observe that the sets in those sequences
may not be distinct.

A set $X \subseteq S$ is 
a \emph{partial transversal}
of the sequence $(A_j \,:\, j \in \indse)$
if there is an injective map
$\varphi \colon X \to J$ such that
$x \in A_j$ if $\varphi(x) = j$.
Those partial transversals are the independent sets of a 
\emph{transversal matroid} $\Mat$
with ground set $\sgrse$.
Observe that  $X \subseteq \sgrse$
is an independent set of $\Mat$
if and only if  there is a matching
in $G$ covering all vertices in $X$.
The sequence $(A_j \,:\, j \in \indse)$
of subsets of $\sgrse$ and, equivalently, the graph $G$
provide a \emph{presentation} 
of the transversal matroid $\Mat$.
A transversal matroid may admit different presentations, 
but there exist presentations such that
the size of $\indse$ equals the rank 
of the matroid~\cite[Theorem~2.6]{Bon10}.
From now on, we always assume that this is the case,
that is, we assume that there is a matching in $G$ 
with $|\indse|$ edges.
In that situation, $B \subseteq \sgrse$ is a basis 
of $\Mat$ if and only if the
subgraph $G_B$ has a  perfect matching. 

The sequence $(C_x \,:\, x \in S)$
of subsets of $\indse$
determines a \emph{Boolean polymatroid} 
with ground set $\sgrse$.
Namely, the polymatroid $(\sgrse,f)$ with
\(
f(X) = |N(X)| =
\left|
\bigcup_{x \in X} C_x 
\right|
\) 
for every $X \subseteq \sgrse$.
By Hall's marriage theorem,
$X \subseteq \sgrse$ is an independent
set of the transversal matroid 
$\Mat$ determined by $G$ if and only if
$|Y| \le |N(Y)| = f(Y)$ for every $Y \subseteq X$.
Therefore, a matroid is transversal
if and only if it is induced by a Boolean polymatroid.

\emph{Lattice path matroids},
which were introduced
in~\cite{BMN03},
are a special class of transversal matroids.
As a consequence of \cite[Lemma~4.7]{Bon10}
the following definition is equivalent
to the one in~\cite{BMN03}.
For positive integers $m,n$, with $m \le n$, we notate
$[m,n] = \{m, m+1, \ldots, n\}$ and
$[n] = [1,n]$. 

\begin{proposition}
\label{st:interv}
Let $G$ be a bipartite graph with parts
$\indse = [r]$ and $\sgrse = [n]$.
Then the following conditions are equivalent.
\begin{enumerate} 
\item
There are sequences 
$(a_1, \ldots, a_r)$ 
and $(b_1, \ldots, b_r)$ in $\sgrse$ with
$1 = a_1 \le a_2 \le \cdots \le a_r$
and $b_1 \le b_2 \le \cdots \le b_r = n$ such that
$a_j \le b_j$ and $A_j  = [a_j,b_j]$ 
for every $j \in \indse$.
\item
There are sequences 
$(c_1, \ldots, c_n)$ and
$(d_1, \ldots, d_n)$ in $\indse$ with
$1 = c_1 \le c_2 \le \cdots  \le c_n$ 
and $d_1 \le d_2 \le \cdots \le d_n = r$
such that $c_x \le d_x$ and 
$C_x = [c_x,d_x]$ for every $x \in \sgrse$.
\end{enumerate}
\end{proposition}

\begin{definition}
A \emph{lattice path matroid}
is a transversal matroid that admits
a presentation with the conditions
of Proposition~\ref{st:interv}.
It is a \emph{nested matroid}
(also called \emph{generalized Catalan matroid})
in the particular case that
it admits such a presentation with
$b_1 = n$ or, equivalently, $c_n = 1$. 
\end{definition}

\subsection{Vector Secret Sharing Schemes}
\label{pt:BrDa}

The reader is referred to~\cite{Bei11} for a
comprehensive survey on secret sharing.
In a \emph{secret sharing scheme}, a 
secret value is distributed into 
\emph{shares} among some \emph{players}
in such a way that only some \emph{qualified} sets of 
players are able to recover the secret from their shares.
The qualified sets form the \emph{access structure},
which is a \emph{monotone} family of sets of players.
That is, every set containing a qualified set is qualified.
A secret sharing scheme is 
\emph{perfect} if the shares from an unqualified set
do not provide any information on the secret value,
and it is \emph{ideal} if,
in addition, each share has the same size 
as the secret value, which is the optimal case.
Brickell and Davenport~\cite{BrDa91} proved that
the access structure of every ideal
secret sharing scheme is a matroid port.

Vector secret sharing schemes are ideal schemes
determined by linear codes.
A \emph{linear code} 
of \emph{length} $n$ 
over a finite field $\kos$ is a vector subspace $C \subseteq K^n$.
The rows of a \emph{generator matrix} form a basis of $C$.
Such a linear code $C$ determines a 
\emph{vector secret sharing scheme} as follows.
Given a secret value $s \in \kos$, choose uniformly at random a
code word $c = (c_1, c_2,\ldots, c_n) \in C$ with $c_1 = s$, 
and distribute the shares $c_2, \ldots, c_n$ among
the $n-1$ players in the scheme.
A set $X$ is qualified
if and only if the first column 
of the generator matrix
is a linear combination of the columns
corresponding to the players in $X$.
Let $\Mat$ be the $\kos$-representable 
matroid associated to the linear code $C$,
that is, the matroid represented by the generator matrix. 
The access structure of the secret sharing scheme
is the port of $\Mat$ at
the element in the ground set corresponding to the first column.
Therefore, the access structures of
vector secret sharing schemes
are the ports of representable matroids.
Each representation of a matroid over a finite field
provides vector secret sharing schemes for its ports
and, conversely, a representation of a matroid
over a finite field is obtained from 
a vector secret sharing scheme
for any of its ports.

\subsection{Matroids with Large Clonal Classes}

Two elements in the ground set
of a matroid are \emph{clones} if the map that
interchanges them and let all other elements fixed
is an automorphism of the matroid.
The equivalence classes of that 
equivalence relation are the \emph{clonal classes} of the matroid,
For example, uniform matroids are those having only one clonal class.

In a secret sharing scheme, 
players $x$ and $y$ are \emph{clones} if, 
for every set $A$ of players with $x,y \notin A$,
the set $A \cup \{x\}$ is qualified
if and only if so is $A \cup \{y\}$.
That is, they play the same role in the scheme.
If the access structure is a matroid port,
two players are clones if and only if 
they are clones in the matroid.

\begin{definition}
A matroid $\Mat$ is \emph{$\Pi$-uniform} for some
partition $\Pi = (\sgrse_i \,:\, i \in \iindse)$ 
of the ground set if
all elements in the same part are clones.
That is, each $\sgrse_i$ is a subset of a clonal class.
If $|\iindse| = m$ , we say that $\Mat$ is 
\emph{$m$-uniform}.
\end{definition}

For a partition $\Pi$ of the set of players,
\emph{$\Pi$-uniform access structures}
are defined analogously.
A secret sharing scheme is said to be 
\emph{multipartite} if its access structure 
is $m$-uniform, specially when $m$ is much smaller
than the number of players.
Ideal multipartite schemes have been studied by several 
authors~\cite{BTW08,Bri89,CTL19,CTL22,FMP07,FaPa10,PaSa00,Sim88,Tas07,TaDy09}.
Their access structures are ports of 
$m$-uniform matroids,
which are called \emph{$m$-partite} in
the works on secret sharing.
The main examples are 
\emph{compartmented} and
\emph{hierarchical} secret sharing schemes.

Let $\Mat = (\sgrse,r)$ be a $\Pi$-uniform
matroid with $\Pi = (\sgrse_i \,:\, i \in \iindse)$.
Associated to $\Mat$, consider the 
integer polymatroid $(\iindse,g)$ with
\[
g(\sindse) = 
r \left(
\bigcup_{i \in \sindse} \sgrse_i
\right) 
\]
for every $\sindse \subseteq \iindse$.
The matroid $\Mat$ 
is determined by the 
integer polymatroid $(\iindse,g)$
and the partition $\Pi$. 
Indeed, consider the map 
$\pi \colon \sgrse \to \iindse$
with $\pi(x) = i$ if $x \in \sgrse_i$
and the polymatroid
$(\sgrse,f)$ with
$f(X) = g(\pi(X))$ for each $X \subseteq \sgrse$.
Then $\Mat$ is the matroid induced 
by the polymatroid $(S,f)$.
The following result was proved
in~\cite[Theorem~6.1]{FMP07}.

\begin{proposition}
Consider a  $\Pi$-uniform matroid 
$\Mat = (\sgrse,r)$   with $\Pi = (\sgrse_i \,:\, i \in \iindse)$
and its associated polymatroid $(\iindse,g)$.
There exists an integer $q(\Mat)$ such that
$\Mat$ is $\kos$-representable
if the field $\kos$ has at least 
$q(\Mat)$ elements and $(\iindse,g)$ is 
$\kos$-representable.
\end{proposition}

Nevertheless, no efficient deterministic methods are known
to find representations of matroids
with large clonal classes
from representations of the associated polymatroids,
which lead to the open problem posed
in~\cite[Open Problem~6.9]{FMP07} 
and~\cite[Section~VII]{FPXY14}. 
Preliminary versions of 
that problem are found
in~\cite{Bri89,Tas07,TaDy09}.
Solutions for some classes of
multi-uniform matroids are given 
in~\cite{BPWX12,Bri89,CRHC24,CTL19,CTL22,Tas07}.

\subsection{Hierarchical Secret Sharing and Lattice Path Matroids}

In an access structure, 
a player $x$ is \emph{hierarchically inferior}
to a player $y$ if, 
for every set $A$
of players with $x,y \notin A$,
the set $A \cup \{y\}$ is qualified if 
so is the set $A \cup \{x\}$.
In that situation, we write $x \preceq y$. 
Observe that $x,y$ are clones if and only if 
$x \preceq y$ and $y \preceq x$.
An access structure is 
\emph{hierarchical} if that preorder in the set of players is total.
\emph{Hierarchical secret sharing schemes}
are those having a hierarchical access structure.

Efficient deterministic constructions of
vector secret sharing schemes 
were presented in~\cite{Bri89,Tas07}
for the so-called 
\emph{hierarchical threshold access structures}.
The construction in~\cite{Bri89} 
was generalized in~\cite{CTL22} to
all hierarchical matroid ports,
which had been previously characterized in~\cite{FaPa10}
in terms of multi-uniform matroids
induced by Boolean polymatroids.
An alternative characterization 
has been recently found~\cite{Mo23},
which is summarized in the following.
If $\Mat$ is a lattice path matroid
on the ground set  $\sgrse = [n]$,
then the ports of $\Mat$ at each of the
elements $p_o = 1$ and $p_o = n$ are
hierarchical access structures.
Conversely, every hierarchical matroid port is of that form.
In particular, hierarchical threshold access structures
are the ports of nested matroids. 
Moreover, the hierarchical order
is compatible with the order in 
the ground set.
Specifically, in the port of $\Mat$ at
$p_o = 1$, a player $x$ is hierarchically inferior
to a player $y$ if $1 < y \le x \le n$
while the hierarchical order is reversed
in the port of $\Mat$ at $p_o = n$. 

Proposition~\ref{st:interv} clarifies
the connection between those two characterizations of
hierarchical matroid ports.
While the characterization in~\cite{FaPa10}
uses the Boolean polymatroid determined
by the sets $C_x$, 
the one in~\cite{Mo23} focuses on the 
lattice path matroid determined by the sets $A_j$.

Therefore, the constructions 
from~\cite{Bri89,Tas07} 
and the ones from~\cite{CTL22}
provide efficient deterministic algorithms to
find representations  over finite fields
for nested matroids
and, respectively, lattice path matroids.
The  method in~\cite{Bri89,CTL22} provides representations
over algebraic field extensions of large degree,
while the one in~\cite{Tas07},
which applies only to nested matroids,
it is based on Birkhoff interpolation
and yields representations over large prime fields.
In the following sections, we give an
alternative description of the former
and we present a new construction over prime fields that
applies to all lattice path matroids.

\section{Representations of Transversal Matroids}


It is well known that representations 
for a transversal matroid $\Mat$ are obtained by
modifying the biadjacency matrix of a presentation $G$.
Indeed,
for each edge $(j,x)$ of $G$,
replace the corresponding entry (equal to $1$) 
in the biadjacency the matrix with a 
variable $\alpha_{j,x}$.
Take an arbitrary field $K$ and assume that 
the entries of the matrix are polynomials
over $\kos$ in the variables
$\alpha_{j,x}$.
Clearly, the determinant of the 
square submatrix formed by
the columns corresponding
to a set $B \subseteq \sgrse$ with $|B| = r$ is
a non-zero polynomial
if $B$ is a basis of $\Mat$ and it is zero otherwise. 
At this point, representations for $\Mat$ 
are obtained  by assigning
values to the variables $\alpha_{j,x}$.
One possibility is considering that
$\alpha_{j,x}$ are algebraically independent
elements over $\kos$ in some extension field.
In addition, for every sufficiently large field $\kos$, 
it is possible to substitute 
the variables $\alpha_{j,x}$
by elements in $\kos$ in such a way that 
the value of every polynomial 
corresponding to a basis of $\Mat$ is non-zero.
Nevertheless, it is not clear how to efficiently choose
those elements.  
We describe next two methods to assign values to
the variables $\alpha_{j,x}$
from a weight function on
the edges of the graph.

\begin{definition}
A weight function 
with non-negative integer values on the
edges of $G$ is \emph{isolating}
if, for every basis
$B$ of the transversal matroid $\Mat$, 
among the perfect matchings of $G_B$
there is only one with minimum weight.
\end{definition}

Every bipartite graph admits an isolating weight function.
Indeed, enumerate the edges
$\{e_0, e_1.\ldots, e_{m-1}\}$
and take $w(e_k) = 2^k$.
Nevertheless, the methods that are described in the following
provide efficient representations for a family of transversal matroids
only if the values of the isolating weight functions are polynomial
in the size of the ground set.
We need the following result about 
the construction of irreducible polynomials
over finite fields.

\begin{proposition}[\cite{Sho90}, Theorems~3.2 and~4.1]
\label{st:shoup}
If $p$ is a prime number, there is 
a deterministic algorithm to find an irreducible polynomial
over $\field_p$ of degree $s$.
Its running time is 
$O(p^{1/2} s^4)$ 
ignoring powers of $\log s$ and $\log p$.
If $q = p^d$, there is a deterministic algorithm
to find an irreducible polynomial over 
$\field_{q}$ of degree $s$ that runs in time
$O(p^{1/2} s^3 + s^4 d^2)$
ignoring powers of $\log s$ and $\log p$.
\end{proposition}

\begin{proposition}
\label{st:reptrm}
Consider a transversal matroid $\Mat$
with rank $r$ over $n$ elements, 
a presentation $G$ of $\Mat$, 
an isolating weight function $w$ on the edges of $G$,
and integers $s,t$ with $t = \max w(j,x)$ and $s$ larger than 
the maximum weight of a matching in $G$.
\begin{enumerate}
\item
There exists a deterministic algorithm
that, for every prime number $p$, 
provides a representation of $\Mat$
over the finite field with $p^{s}$ elements.
The running time
is polynomial in $n$, $p$ and $s$.
\item
There exists a deterministic algorithm
that, for each prime number 
$p > 2^{rt} \, r^{r/2}$,
provides a representation of $\Mat$
over $\field_p$.
The running time of the algorithm is
polynomial in $\log p$ and $n$.
\end{enumerate}
\end{proposition}

\begin{proof}
Take a variable $\alpha$ and put
$\alpha_{j,x} = \alpha^{w(j,x)}$
for every edge $(j,x)$ of $G$.
Take an arbitrary field $\kos$ and
assume that the entries of the matrix are polynomials
over $\kos$ in the variable $\alpha$.
Then the determinant of the submatrix corresponding
to a basis $B$ is a non-zero polynomial.
Indeed, the coefficient of the minimum degree 
equals either $1$ or $-1$
because it corresponds to the unique perfect matching in $G_B$ 
with minimum weight.
For each basis, the degree of that polynomial 
is at most the maximum weight of
a perfect matching, and hence less than $s$.

Consider an arbitrary prime number $p$ and $q = p^s$,
and take $\kos = \field_p$.
By using the algorithm given by 
Shoup~\cite{Sho90}, 
find an irreducible polynomial 
$f(\alpha)$ over $\field_{p}$ of degree $s$.
As we mentioned in Proposition~\ref{st:shoup},
that can be done  in time $O(p^{1/2} s^4)$
ignoring the powers of
$\log s$ and $\log p$.
Then the quotient ring
$\field_p[\alpha]/(f(\alpha))$
is isomorphic to the field $\field_q$,
an algebraic extension of $\field_p$.
The class of $\alpha$ in that quotient ring
is an element in $\field_q$
whose minimal polynomial over $\field_{p}$ is of degree $s$.
By identifying the entries of the matrix
with elements in $\field_q$, a representation
of the matroid $\Mat$ over that field is obtained.

In order to construct a 
representation over a prime field,
assume that $K$ is the field of real numbers,
and hence the entries of the matrix are assumed to be
real polynomials in the variable $\alpha$.
Those polynomials have integer coefficients
because they are either zero or a power of $\alpha$.
Therefore, for every basis $B$, 
the determinant of the corresponding submatrix
is a non-zero polynomial $h_B(\alpha)$ with integer coefficients.
Moreover, $h_B(2) \ne 0$ because
the coefficient of the minimum degree term is  $\pm 1$.
Put $\alpha = 2$ and let $A$ be the the resulting integer matrix.
Observe that $0 \le a_{j,x} \le 2^t$ for every entry of that matrix. 
For a basis $B$, the corresponding submatrix $A_B$ satisfies
\[
|\det A_B| \le 2^{rt}\, r^{r/2}
\]
by Hadamard's inequality.
Therefore, that determinant is
not a multiple of $p$
for any prime number 
$p$ larger than $2^{rt}\, r^{r/2}$,
and hence the matrix $A$ provides a 
representation of $\Mat$ over
the prime field $\field_p$.
\end{proof}

The most computationally expensive step in
the first algorithm is to find an irreducible
polynomial over $\field_{p}$ of degree $s$.
Since $p$ can be the same for all matroids in the family,
the computation time depends 
almost exclusively on the value of $s$,
and hence on the maximum
weight of the perfect matchings in the subgraphs $G_B$.
In addition, the value of $s$ determines the size
of the representations, and hence the efficiency of 
their applications as, for example, secret sharing schemes. 
The size of the representations given by the second construction
is $n r \log p$,
and hence it provides representations of size
$O(n r( r t + (r/2)\log r))$,
where $t$ is the maximum weight of the edges.

\section{Efficient Representations of Lattice Path Matroids}

In this section, we consider only
transversal matroids of rank $r$ with a representation
$G$ with $\indse = [r]$ and $\sgrse = [n]$.
For a set $B\subseteq \sgrse$, 
we notate $B = (x_1, \ldots, x_r)$
to indicate that its elements
are arranged in increasing order.  

We present in Proposition~\ref{st:suff}
a sufficient condition for the existence of
isolating weight functions with polynomial weights.
By combining it with Proposition~\ref{st:reptrm}, 
efficient representations
for lattice path matroids are obtained.
The following technical result is a consequence of
the \emph{rearrangement inequality}.

\begin{lemma}
\label{st:easy}
Let $(p_1, \ldots, p_r)$ and
$(q_1, \ldots, q_r)$  be sequences
of real numbers such that the first one is 
non-decreasing and 
the second one is non-increasing.
Then 
\[
p_1 q_1 + \cdots + p_r q_r \le
p_1 q_{\sigma 1} + \cdots + p_r q_{\sigma r} \le
p_1 q_{r} + \cdots + p_r q_{1}
\]
for every permutation $\sigma$.
Moreover, each of those bounds is 
attained only by one permutation  
if each sequence has distinct terms.
\end{lemma}

\begin{proposition}
\label{st:suff}
Let $\Mat$ be a transversal matroid 
such that, for each basis 
$B = (x_{1}, \ldots, x_{r})$, 
all pairs $(j,x_j)$ with $j \in \indse$
are edges of $G$.
Then $G$ admits an
isolating weight function with
maximum weight at most $(r-1)(n-1)$.
In addition, for each basis $B$,
the maximum weight of
the perfect matchings in $G_B$ is less than
\(
r (r-1)(n-1)/2
\).
\end{proposition}

\begin{proof}
For $j \in [r]$ and $x \in [n]$, 
take $p_j = j - 1$ and $q_x = n - x$. 
For every edge $(j,x)$, take the weight
$w(j,x) = p_j q_x$.
This is an isolating weight function
because, by Lemma~\ref{st:easy},
the perfect matching formed
by the edges $(j, x_{j})$ 
is the only one in $G_B$ with minimum weight.
Finally, by Lemma~\ref{st:easy} again,
the weight of a perfect matching in $G_B$ is at most 
\begin{equation}
\label{eq:bound}
(r-1)(n-1) + (r -2)(n-2) + \cdots + 1 \cdot(n-r+1)
\end{equation}
and hence less than
\(
r (r-1)(n-1)/2
\).
Smaller upper bounds can be
obtained from~(\ref{eq:bound}),
but they are not better than $O(r^2 n)$.
\end{proof}

By the following two propositions,
lattice path matroids
are the only transversal matroids
satisfying the sufficient condition in 
Proposition~\ref{st:suff}. 

\begin{proposition}
\label{st:charlpm1}
Let $\Mat$ be a lattice path matroid
and let $G$ be a presentation of $\Mat$
in the conditions of Proposition~\ref{st:interv}. 
If $B = (x_{1}, \ldots, x_{r})$ is a basis of $\Mat$, then
$(j, x_{j})$ is an edge of $G$ for every $j \in \indse$.
\end{proposition}

\begin{proof}
Suppose that there is a basis 
$B$ without that property.
Let $P$ be a perfect matching in $G_B$ with the
maximum number of edges of the form $(j,x_j)$
and take the minimum $k \in J$ such that
$(k,x_k)$ is not in $P$.
Since $P$ is a perfect matching, $k \le r-1$ and 
there exist $\ell_1, \ell_2 \in [k+1,r]$ such that 
$(\ell_1,x_{k})$ and $(k,x_{\ell_2})$ 
are edges in $P$.
Then 
\[
a_{k} \le a_{\ell_1} \le x_{k} 
< x_{\ell_2} \le b_{k} \le b_{\ell_1}
\]
which implies that $(k,x_{k})$ and 
$(\ell_1,x_{\ell_2})$ are edges of $G$.
Then
\[
P' = (P \sss \{(k, x_{\ell_2}), (\ell_1,x_k)\})
\cup \{(k,x_k),(\ell_1, x_{\ell_2})\}
\]
is a perfect matching in $G_B$ with more edges 
of the form $(j,x_j)$ than $P$.
\end{proof}

\begin{proposition}
\label{st:charlpm2}
Let $\Mat$ be a  transversal matroid without loops
that admits a presentation 
$G$  such that, 
for every basis 
$B = (x_{1}, \ldots, x_{r})$
and for every $j \in \indse$, 
the pair $(j, x_{j})$ is an edge.
Then $\Mat$ is a lattice path matroid.
\end{proposition}

\begin{proof}
Let $B_1 = (a_1, \ldots, a_r)$
and $B_2 = (b_1, \ldots, b_r)$ be the
first and last bases of $\Mat$ 
in the lexicographic order.
We are going to prove that the sequence of sets
$([a_j,b_j] \, :\, j \in \indse)$ is a presentation of $\Mat$,
and hence it is a lattice path matroid.
For two distinct bases $B,B'$,
we notate $B \ll B'$
if $B$ precedes $B'$ in the lexicographic order.
We prove first that
$x_j  \in [a_j,b_j]$ for each $j \in [r]$
if $(x_1, \ldots, x_r)$ is a basis.
Suppose that there is a basis with $x_j < a_j$ 
for some $j \in [r]$.
Take $B$ the first such basis 
in the lexicographic order
and the minimum $j \in [r]$ with $x_j < a_j$.
Since $B_1 \ll B$, the minimum  $k \in [r]$ with
$x_{k} > a_{k}$ satisfies $k < j$. 
Then $B' = (B \sss \{x_{k}\}) \cup \{a_{k}\}$
is another basis in the same situation with
$B' \ll B$, a contradiction.
Symmetrically, $x_j \le b_j$ for each $j \in [r]$.
We prove next that
$(j,x)$ is an edge if $x \in [a_j,b_j]$.
Since $x$ is not a loop, 
there is an edge $(k,x)$.
If $k > j$ and $x \ne b_j$, consider the basis
$B = (a_1, \ldots, a_{j-1}, b_j, \ldots, b_r)$.
Then $(B \sss \{b_k\}) \cup \{x\}$ is a basis
and $x$ is its $j$-th element, 
which implies that $(j,x)$ is an edge.
Symmetrically, the same happens if $k < j$ and $x \ne a_j$.
\end{proof}

The following result is a direct consequence of
Propositions~\ref{st:reptrm},~\ref{st:suff}, 
and~\ref{st:charlpm1}.

\begin{proposition}
\label{st:replpm}
Given a presentation of a lattice path matroid $\Mat$
with rank $r$ on $n$ elements
in the conditions of Proposition~\ref{st:interv},
there are two efficient constructions of 
representations for $\Mat$, which are described in the following.
\begin{enumerate}
\item
There is an efficient deterministic
algorithm to find a representation of $\Mat$
over the finite field with $q = p^s$ elements, where 
\(
s = r(r-1)(n-1)/2
\)
and $p$ is an arbitrarily chosen prime number.
The running time of the algorithm is
polynomial in $p$ and the size $n$ of the ground set.
\item
For every prime number
$p > 2^{r(r-1)(n-1)} r^{r/2}$,
there is an efficient deterministic
algorithm to find a representation of $\Mat$
over $\field_p$.
The running time of the algorithm is
polynomial in $\log p$ and the size $n$ of the ground set.
\end{enumerate}
\end{proposition}

The construction of hierarchical threshold 
secret sharing schemes in~\cite{Tas07},
which uses Birkhoff interpolation, provides another 
efficient deterministic algorithm 
to find representations
of nested matroids over prime fields 
$\field_p$ with
\[
p > (r-1)! \; 2^{-r+2} (r - 1)^{(r-1)/2} n^{(r-1)(r-2)/2}  
\]
The second construction in Proposition~\ref{st:replpm}
yields less efficient representations, 
but it applies to all lattice path matroids.

We present next an improvement to the first algorithm
in Proposition~\ref{st:replpm} for
lattice path matroids with 
a relatively small number of clonal classes.
It is equivalent to the constructions of hierarchical
vector secret sharing schemes
in~\cite{Bri89,CTL22}.

Take $\indse = [r]$, $\sgrse = [n]$, and
integers $t_i$ with 
$1 = t_1 < t_2 < \cdots < t_m < t_{m+1}= n + 1$.
Consider the partition 
$\Pi = (\sgrse_i \,:\, i \in [m])$
of $S$ with $S_i = [t_i, t_{i+1} - 1]$.
For every $x \in \sgrse$, 
put $\pi(x) = i$ if $x \in S_i$.
Consider a bipartite graph $G$ in the 
conditions of Proposition~\ref{st:interv}
such that, for each $i \in [m]$, 
all vertices in $S_i$
have the same neighbors.
Then $G$ is a presentation of a 
$\Pi$-uniform lattice path matroid $\Mat$.
Observe that the port of $\Mat$ 
at the element $1 \in \sgrse$
is a hierarchical access structure in which 
all players in the same part are
hierarchically equivalent.

As we did before, we replace the non-zero 
entries of the biadjacency matrix of $G$
with polynomials in the variable $\alpha$
over some finite field.
Take a prime power $q$ such that
$q > |S_i| = t_{i+1} - t_i$ for every $i \in [m]$.
For each $i \in [m]$, take
$t_{i+1} - t_i$ distinct non-zero elements
$(\beta_{x} \,:\, x \in S_i)$ in the finite field $\field_q$.
For $j \in \indse$ and $x \in \sgrse$, take $p_j = j-1$
and $q_x = m - \pi(x)$, and
consider on the edges of $G$ the weight function 
$w(j,x) = p_j q_x$.
Finally, consider the 
matrix $H$ that is obtained by
replacing the entry in the biadjacency matrix of $G$
corresponding to the edge $(j,x)$ 
with $\beta_x^{j-1} \alpha^{w(j,x)}$.

We prove next that, for every basis $B$ of $\Mat$,
the determinant of the submatrix $H_B$ 
formed by the corresponding columns
is a non-zero polynomial.
Even though the chosen weight function is not isolating,
we can check that the
coefficient of the minimum degree term is non-zero. 
Indeed, let $B = (x_1, \ldots, x_r)$ be a basis of $\Mat$. 
By Lemma~\ref{st:easy}, 
the perfect matching $((j,x_j) \,:\, j \in \indse)$
has minimum weight, but 
there are other perfect matchings 
in $G_B$ with the same weight,
namely the ones of the form
$((j,x_{\sigma j}) \,:\, j \in \indse)$,
where $\sigma$ is any permutation
such that $\pi(x_{\sigma j}) = \pi(x_j)$
for every $j \in \indse$.
The entries corresponding to the edges
of $G_B$ involved in those perfect
matchings lie on square submatrices
on the diagonal of $H_B$,
one for each $i \in [m]$ with $B \cap S_i \ne \emptyset$.
The determinant of each of those submatrices is of the form
$\alpha^{\ell_i} \Delta_i$, where $\Delta_i$ is
the determinant of a Vandermonde-like matrix, and hence non-zero.
Therefore, the coefficient of the minimum degree term
of $\det H_B$ is equal to $\prod_i \Delta_i \ne 0$.
Observe that the weight of a perfect matching
in any subgraph $G_B$ is less than
$r(r-1)(m-1)/2$.
At this point, the following result has been proved.

\begin{proposition}
\label{st:replpmclon}
There exists a deterministic algorithm
that, given an $m$-uniform 
lattice path matroid $\Mat$ with the conditions above,  
provides a representation of $\Mat$
over a finite field with $q^s$ elements,
where $q$ is a prime power larger than the number of 
elements in each part
and $s = r(r-1)(m-1)/2$.
The running time of the algorithm is
polynomial in $q$ and the size $n$ of the ground set.
\end{proposition}

This algorithm improves on the first one in 
Proposition~\ref{st:replpm} if the number of parts $m$ is 
small in relation to the size of the ground set.
Even though $q$ cannot be arbitrarily small, 
the degree $s$ of the extension can be much smaller and,
as we discussed before, this is the main 
parameter to be taken into account.

Every bi-uniform matroid (that is, $m=2$) is a lattice path matroid,
and hence the algorithm in Proposition~\ref{st:replpmclon}
provides representations with $s = r(r-1)/2$.
Nevertheless, the algorithm proposed in~\cite{BPWX12}
is in general more efficient 
because the degree of the extension is
$s = d (d-1)/2$, where $d = r(S_1) + r(S_2) - r$.

\section*{Acknowledgment}
Thanks to Anna de Mier for an enlightening discussion
about transversal matroids.
The author's work was supported by 
the Spanish Government
under projects PID2019-109379RB-I00
and PID2021-124928NB-I00.


\end{document}